\documentclass[preprint,fleqn,12pt]{elsarticle}

\usepackage{algorithm}

\usepackage{algorithmic}

\usepackage{multirow}

\usepackage{amsmath, amsthm}

\usepackage{xcolor}
\usepackage{amssymb}
\usepackage[bookmarks]{hyperref}
\usepackage[titletoc]{appendix}
\newtheorem{thm}{Theorem }[section]
\newtheorem{pro}{Proposition}[section]

\newtheorem{example}{Example}[section]

\journal{J. Computational and Applied Mathematics}

\begin{document}

\begin{frontmatter}



\title{Numerical method for real root isolation of \emph{semi-algebraic} system and its applications
\tnoteref{t1}} \tnotetext[t1]{This research was partially supported
by the National Natural Science Foundation of China(11171053), and the National Natural Science Foundation of China Youth Fund Project(11001040,61103110)}


\author{Zhenyi Ji\corref{cor1}$^{1,2}$}\ead{zyji001@163.com}\author{Wenyuan Wu$^{2}$}\author{Yi Li$^{2}$}\author{Yong Feng$^{2}$}

\address{$^{1}$Lab. of Computer Reasoning and
Trustworthy Comput, School of Computer Science and Engineering, University of Electronic Science and Technology
of China, Chengdu 611731, P.R.China}

\address{$^{2}$Lab. of Automated Reasoning and Cognition, Chongqing Institute of Green and Intelligent Technology,
 Chinese Academy of Science, Chongqing, 401120, P.R. China}\cortext[cor1]{Corresponding author}

\address{}

\begin{abstract}
In this paper, based on the homotopy continuation method and the interval Newton method, 
 an efficient algorithm is introduced to isolate the real roots of \emph{semi-algebraic} system. 
 Tests on some random examples and a variety of problems including transcendental functions arising in many applications show
 that the new algorithm reduces the cost substantially  compared with the traditional symbolic approaches.
\end{abstract}


\end{frontmatter}


\section{Introduction}

The problem of counting and isolating real solutions of nonlinear system is an important topic in computing geometry and many other applications in various fields, e.g.,
  the real intersection points for piecewise algebraic curve \cite{WZ2008,ZW2009,W2012,LW2009,WW2008}, the stability of a large class of biological
  networks \cite{AFJS2004,WX2005}, discovering non-linear ranking functions of loop programs in \cite{CXYZZ2007,XYZ2008}, automated proving inequality type theorem \cite{YHX2001}
   and so on.

Many algorithms for real root isolation of one polynomial in one variable have been developed in \cite{CL1982,ABS1994,AS2005,ZX2007,MRR}. For multivariable case,
Xia et al.\cite{XL2002} proposed an algorithm based on Wu's method for isolating the real roots of \emph{semi-algebraic} system with integer coefficients,
and made it more available with interval algorithm in their later work \cite{XZ2006}. There are also other algorithms based on different techniques,
 see \cite{BFLM2009,CGG2012,CGY2007,R1999} for more details.

Actually, most of the algorithms mentioned above can compute the exact results because they depend on symbolic computations, but they are restricted to small size systems because
of the high complexity of the symbolic computation. In order to avoid this problem, Shen et al.\cite{S2012} presented a numerical algorithm improving the efficience based on homotopy
continuation method combined with interval Newton iteration technique.

In this paper, we extend the numerical method to a class of \emph{semi-algebraic} systems and transcendental functions.  We always assume that the system is square and it
only has isolated roots in $\mathbb{C}$. In order to avoid the singularity of the Jacobian matrix, we also suppose that the multiplicity of these points are one.

The rest of the paper is organized as follows.  Definitions and
preliminaries about homotopy continuation method and interval arithmetic are given in section 2. Section 3 first gives a brief review of  the numerical algorithm for isolating real roots
of a class of polynomial systems, and then presents the numerical method for isolating the real roots of \emph{semi-algebraic} systems. The  experimental results  together with
comparison to symbolic methos is given in this section. Applying the new algorithm to some problems arising from piecewise algebraic curve, chemical engineering, robot kinematic problem, circuit
design are shown in section 4.
Finally, section 5 draws a conclusion of this paper.

\section{Interval Arithmetic And Homotopy Continuation Method }

In this section, some basic theories and tools about interval arithmetic and homotopy continuation method are presented.

\subsection{Interval Newton's method}

Interval operations were first introduced by Moore \cite{M1966}. The key idea of interval arithmetic is substituting an interval for
a floating number in numerical computation, which is used to tackle the instability and error analysis. For a
more detailed and complete discussion the reader is referred to \cite{W1987}.

A real \textbf{interval} $X$ is a nonempty set of real numbers
$$X=[\underline{x},\bar{x}]=\{x \in \mathbb{R}: \underline{x}<x<\bar{x}\},$$
the set of all intervals in $\mathbb{R}$ is denoted by $I(\mathbb{R})$.

An \emph{interval vector} \textbf{X} is  a vector whose elements are intervals, and an \emph{interval matrix} can be similarly defined. In the rest of this paper,
we let \textbf{\texttt{X}} denotes the set of \emph{interval vector}.

For an interval $X$, the \textbf{midpoint} of  $X$ is $m(X)=(\underline{x}+\bar{x})/2$, the\textbf{ width} of $X$ is $w(x)=\bar{x}-\underline{x}$,
and the \textbf{radius} of $X$ is $r(X)=(\bar{x}-\underline{x})/2$.

Given interval $X=[\underline{x},\bar{x}]$, $Y=[\underline{y},\bar{y}]$, the four element operations are defined  as \\
\vspace{+0.1cm}
$X+Y=[\underline{x}+\underline{y},\bar{x}+\bar{y}]$\\
\vspace{+0.1cm}
$X-Y=[\underline{x}-\bar{y},\bar{x}-\underline{y}]$\\
\vspace{+0.1cm}
$X\cdot Y=[min\{\underline{x}\underline{y},\underline{x}\bar{y},\bar{x}\underline{y}, \bar{x}\bar{y}\},max\{\underline{x}\underline{y},\underline{x}\bar{y},\bar{x}\underline{y}, \bar{x}\bar{y}\}]$\\
\vspace{+0.1cm}
$X/Y=[\underline{x},\bar{x}]\cdot[1/\underline{y},1/\bar{y}]$, $0\notin[\underline{y},\bar{y}]$.

The intersection of two intervals $X$ and $Y$ is empty if either $\bar{x}<\underline{y}$ or $\bar{y}<\underline{x}$. In this case
, we write $X \cap Y=\emptyset .$

For interval matrices and interval vectors, the
concepts such as midpoint, width, radius, etc, and the arithmetic operations are
defined in components.

Let $\mathbb{R}[\textbf{\emph{x}}]$ be the ring of polynomials in the variables $(x_1,x_2,\cdots,x_n)$ with coefficients in $\mathbb{R}$, and \emph{\textbf{f}}$=[f_1,\cdots,f_n]$ be a polynomial system, where $f_i\in \mathbb{R}[\textbf{\emph{x}}]$.

Suppose $f: \mathbb{R}^{n}\rightarrow \mathbb{R}$  is a function from a real vector to a real number, $F: I(\mathbb{R}^{n})\rightarrow I(\mathbb{R})$  is called an
\emph{ interval extension} of $f$ if
$$F([x_1,x_1],\cdots,[x_n,x_n])=f(x_1,\cdots,x_n)$$ for all $x_i \in X_i, i=1,2,\ldots,n.$

We say that $F = F(X_1,\cdots, X_n)$ is \emph{inclusion monotonicity }if
$$Y_i \subseteq X_i,  i=1,\cdots, n \Rightarrow F(Y_1,\cdots, Y_n) \subseteq F(X_1,\cdots,X_n).$$
And it is easy
to prove that all the polynomial operations satisfy the inclusive monotonicity.

In the following, we let \emph{\textbf{f}}$'$ be the Jacobian matrix of \emph{\textbf{f}} , and $\emph{\textbf{F}}$, \emph{\textbf{F}}$'$ be the \emph{interval extension}
of \emph{\textbf{f}} and \emph{\textbf{f}}$'$ with inclusive monotonicity, respectively.

 Moore first define the following interval Newton's operators:
\begin{equation}
N(\textbf{X})=m(\textbf{X})-V(\textbf{X})\emph{\textbf{f}}\text{ }{(m(\textbf{X}))},
\end{equation}
where $V(\textbf{X})$ is an \emph{interval matrix} containing \emph{\textbf{F}}$'$(\textbf{X})$^{-1}$.

In order to avoid the computation of the inversion of interval matrix in formula (1), Krawcayk  proposed the following  operator:
\begin{equation}
K(\textbf{\emph{y}},\textbf{X})=\textbf{\emph{y}}-\textsf{Y}\textbf{\emph{f}}(\textbf{\emph{y}})+(\textbf{I}-\textsf{Y}\emph{\textbf{F}}\text{ } '(\textbf{X}))(\textbf{X}-\textbf{\emph{y}}),
\end{equation}
where $\textbf{\emph{y}}$ is chosen from the region $\textbf{X}$, \textbf{I} denotes the \emph{unit matrix}, 
and \textsf{Y} is an arbitrary nonsingular matrix.

It has been proved that \emph{Krawcayk  operator } has the following properties.

\begin{pro}\label{pro21}
Suppose $K(\textbf{{y}},\textbf{X})$ is computed by formula (2), then

1): $\textbf{x}^{*}\in \textbf{X} \text{ and }$ \textbf{f}\text{ }$(\textbf{x}^{*})=0 \Rightarrow \textbf{x}^{*}\in K(\textbf{{y}},\textbf{X})$.

2): $K(\textbf{{y}},\textbf{X}) \subset int(\textbf{X})\Rightarrow$ \textbf{{f}} has only one root in $\textbf{X}$, where $int(\textbf{X})$ denotes the topological interior of the box of $\textbf{X}$.

3): $K(\textbf{{y}},\textbf{X}) \subseteq int(\textbf{X})\Rightarrow$ \textbf{{f}} has a root in $\textbf{X}$.

4): $K(\textbf{{y}},\textbf{X})\cap \textbf{X}=\varnothing \Rightarrow$ no solution in $\textbf{X}$.
\end{pro}

In particular, if \textbf{\emph{y}} and \textsf{Y} are chosen to be $\textbf{y}=m(\textbf{X})$ and $\textsf{Y}=[m($\emph{\textbf{F}}$'(\textbf{X}))]^{-1}$ respectively, then the Moore form of the \emph{Krawcayk  operators } is :
\begin{equation}
K(\textbf{X})= m(\textbf{X})-[m(\emph{\textbf{F}}\text{ } '(\textbf{X}))]^{-1}\textbf{\emph{f}}(m(\textbf{X}))+\Delta,
\end{equation}
where $\Delta=(\textbf{I}-[m(\emph{\textbf{F}}\text{ } '(\textbf{X}))]^{-1}\emph{\textbf{F}}\text{ } '(\textbf{X}))(\textbf{X}-m(\textbf{X})).$

\subsection{Homotopy continuation method}

Homotopy continuation method is an efficient numerical method for finding all isolated solutions of polynomial system. The method traces a path from the solution of an
easy problem to the solution of the given problem by use of a homotopy continuous
transformation 
. See reference \cite{L1999,S1986,VVC1994,L2003,SW2005} for more details.
There also exists some software packages \cite{Llll2008,V1999,W2005} for homotopy continuation methods. In our implementation, we use Hom4ps-2.0 which could return all the approximate complex zeros of a given polynomial system efficiently, along with
residues and condition numbers.

\section{Algorithm for isolating real zeros of \emph{semi-algebraic} system}

In this section, we will present an algorithm for isolating the real roots of a zero-dimensional \emph{semi-algebraic} system. Our idea is to compute the isolated real root intervals
of the zero-dimensional polynomial system first. In this step, we will use the hybrid algorithm introduced by Shen et, al \cite{S2012}. Then decide whether these intervals
satisfy inequations by substitution.

First, we introduce the definition of \emph{semi-algebraic} system. A \emph{semi-algebra} system denoted by $SAS$ can be written in the following form:
\begin{eqnarray}
\left\{ {\begin{array}{*{20}{c}}
 f_1=0,\cdots,f_n=0, \hfill \\
 p_1>0,\cdots,p_q>0 ,\hfill  \\
 n_1\geq 0,\cdots,n_s\geq0, \hfill\\
 h_1\neq0,\cdots,h_t \neq 0, \hfill\\
  \end{array}} \right.
\end{eqnarray}
where $n\geq1$ and $q,s,t\geq0$. We call it zero-dimensional \emph{semi-algebraic} system if $\{f_1,\cdots,f_n\}$ has only finite zeros in $\mathbb{C}$. Following the notations in
software package Discover \cite{XL2002},
we let \textbf{\emph{f}},\emph{\textbf{ n}}, \textbf{\emph{p}}, \textbf{\emph{h}} denote the polynomial equations, non-negative polynomial inequalities,  positive polynomial inequalities and polynomial inequations respectively.

\subsection{Real root isolation for zero-dimensional polynomial system}
In this subsection, we introduce a hybrid method for the real root isolation of a zero-dimensional polynomial system, see \cite{S2012} for more details.

Suppose we have obtained all the isolated approximate roots of \emph{\textbf{f}} through homotopy continuation method.
For each approximate root, the following proposition gives a method  to construct initial interval which contains its corresponding accurate root.
\begin{pro}\label{pro31}\cite{S2012}
Let \textbf{f}$ \text{}=[f_1,\cdots,f_n]$ be a polynomial system in ${\mathbb{R}[\textbf{{x}}}]$, and $\bar{\textbf{{x}}} \in \mathbb{R}^n$ be an
approximate zero of \textbf{f}. If the following conditions hold:

1: $\textbf{f}\text{ }'(\bar{\textbf{{x}}})$ exists, and there are real numbers $B$ and $\eta$ such that

$$||\textbf{f}\text{ }'(\bar{\textbf{{x}}})^{-1}||\leq B, ||\textbf{f}\text{ }'(\bar{\textbf{{x}}})^{-1}\textbf{f}(\bar{\textbf{{x}}}) ||\leq \eta,$$

2: There exists a ball neighbourhood $ O(\bar{\textbf{{x}}},\omega)$ such that $\textbf{f}\text{ }'(\textbf{{x}})$ satisfies the Lipschitz condition on it:

$$||\textbf{f}\text{ }'({\textbf{{x}}})-\textbf{f}\text{ }'({\textbf{{y}}})||\leq {K}||{\textbf{{x}}}-{\textbf{{y}}}||, \forall {\textbf{{x}}},{\textbf{{y}}} \in O(\bar{\textbf{{x}}},\omega),$$

3: $$h=BK\eta\leq \frac{{1}}{2}, \omega \geq \frac{{1-2\sqrt{1-2h}}}{h}\eta,$$
then \textbf{f} has only one root in $\overline{O(\bar{\textbf{{x}}},\omega)}.$
\end{pro}

 Let $\hat{\textbf{\emph{x}}}$ be an accurate root, and $\bar{\textbf{\emph{x}}}$ be its approximation. Denote the \emph{Hessian matrix} of $f_j$ by $H_j$,
and $h_j^{i}$ its column vector. Let
\begin{equation}
\lambda=max_{1\leq i\leq n} \sum\limits_{j=1}^{n}|h_j^{i}(\bar{\textbf{\emph{x}}})|_{max},\text{ and } r=\frac{{||J_F^{-1}(\bar{\textbf{\emph{x}}})||_{\infty} ||J_F(\bar{\textbf{\emph{x}}})||_{\infty}^{2}}}{1-n\lambda||J_F^{-1}(\bar{\textbf{\emph{x}}})||_{\infty} ||J_F(\bar{\textbf{\emph{x}}})||_{\infty}},
\end{equation}
where $|\cdot|_{max}$ denotes the maximum module component of a vector.

A natural fact about $r$ is that if  $r>Im(\bar{\textbf{\emph{x}}})$, then
$\bar{\textbf{\emph{x}}}$ is a imaginary number, where $Im(\bar{\textbf{\emph{x}}})$ means the imaginary part of $\bar{\textbf{\emph{x}}}$. So we can delete some complex roots without using the interval arithmetic. 


We use algorithms from \cite{S2012} to compute $\omega$ and $\lambda$.

According to the above descriptions, the following algorithm $real_{-}roots$ can be used for isolating real root intervals of \emph{\textbf{f}}.

\noindent \textbf{Algorithm 1 }:  $real_{-}root_{-}isolate$ \cite{S2012}

Input: Polynomial system: \emph{\textbf{f}}, and a threshold $\tau$.

Output: Isolated intervals of \emph{\textbf{f}} $real_{-}roots$ or $\{\}$.

\begin{enumerate}
\item Let $real_{-}roots=\{\}.$
\item Computing all the isolated roots of \emph{\textbf{f}}, denote it by $roots$, and let $nreal$ be the number of $roots$.
\item If $nreal=0$, then  return $real_{-}roots$, and stop this program.
\item For $i=1:nreal$;
      \begin{itemize}
      \item let $z$ be the $i$th element of $roots$, and determine whether $z$ is a complex root using formula (5) and $Im(z)$.
      \item If $z$ is a complex root, stop; else, construct the initial interval of the real part of $z$ according to proposition 3.1, and denote it by  $\textbf{X}_0$.
      Compute $K(\textbf{X}_0)$ according to formula (3). If $K(\textbf{X}_0)\subset \textbf{X}_0$, then let  $real_{-}roots=real_{-}roots \cup \{K(\textbf{X}_0)\cap \textbf{X}_0\}$, and stop this step. If $K(\textbf{X}_0)\cap \textbf{X}_0=\varnothing $, stop. Otherwise, bisect $\textbf{X}_0=K(\textbf{X})\cap \textbf{X}_0$, and process each half separately.
      \end{itemize}
\item End For
\item If there exists some intervals in $real_{-}root$ such that they are not disjoint, then compute the intersection of these intervals,
split each interval $\textbf{X}_i=\textbf{X}_{i1} \cup \textbf{X}_{i2}$, where $\textbf{X}_{i1}$ denotes the intersection portion,
 process all  the subintervals,
 and remove the subinterval which don't contain a real root of \emph{\textbf{f}}.
\item If the width of an interval is bigger than $\tau$, then bisect the interval and process each half separately.
\item Return $real_{-}roots$.
\end{enumerate}

\subsection{Algorithm for {semi-algebraic} system}
 Now suppose we have obtained the isolated  real root intervals of polynomial system \textbf{\emph{f}} defined in (4).
Let $\textbf{X}=([a_1,b_1],\cdots,[a_n,b_n])$ be an isolated real root interval of \textbf{\emph{f}},
and  $\hat{\textbf{\emph{x}}}$ be
  the accurate zero of \textbf{\emph{f}} such that $\hat{\textbf{\emph{x}}} \in \textbf{X}$. Assume $f$ is a positive polynomial inequality
   in (4) and $f(\textbf{X})=[a,b]$, then there are only three cases that can happen about the sign of $[a,b]$, including the
following:

1: $a>0  \Rightarrow f(\hat{\textbf{\emph{x}}})>0$.

2: $ b< 0 \Rightarrow f(\hat{\textbf{\emph{x}}})<0$.

3: $0 \in [a,b]$.

Hence, if $[a,b]$ satisfies the first case or the second, we can easily decide whether $\textbf{X}$ satisfy $f$ or not.
For case 3, it is difficult to decide the sign of $f(\hat{\textbf{\emph{x}}})$.  In the following, we present a method to solve this problem.

Let  $g=y^2f+1$ be a new polynomial in $\mathbb{R}[\textbf{\emph{x}},y]$, and $[\emph{\textbf{f}},g]$ be a  polynomial system
$\{f_1,\cdots,f_n,g\}$ in $\mathbb{R}[\textbf{\emph{x}},y]$,  where $\mathbb{R}[\textbf{\emph{x}},y]$ denotes the ring of polynomials in the variables
$(x_1,x_2,\cdots,x_n,y)$ with coefficients in $\mathbb{R}$. Assume $\textbf{\texttt{Z}}=\{\textbf{Z}_1,\cdots,\textbf{Z}_m\}$ is the set of the isolated
real root intervals of $[\emph{\textbf{f}},g]$. Define the projection $\pi: I(\mathbb{R}^{n+1})\rightarrow I(\mathbb{R}^{n})$ which remove the last coordinate of an interval vector in $I(\mathbb{R}^{n+1})$.
\begin{thm}\label{thm32}
If the intersection of $\textbf{X}$ and $\pi(\textbf{Z}_i)$ is an empty set for all $i\in \{1,2,\cdots,m\}$, then $f(\hat{\textbf{{x}}})\geq0$.
\end{thm}

\begin{proof}Denote $\emph{\textbf{X}}$ by $ ([a_1,b_1],\cdots,[a_n,b_n])$ and $\textbf{\emph{Z}}_i=([a_{i1},b_{i1}],\cdots,[a_{in},b_{in}],$ \\ $[y_{i1},y_{i2}]),i=1,2,\cdots,m$.

 Suppose $f(\hat{\textbf{\emph{x}}})<0$, then it is easy to see that the point $(\hat{\textbf{\emph{x}}},  \sqrt{-1/f(\hat{\textbf{\emph{x}}})})$ is the accurate zero of [\textbf{\emph{f}}$,g]$,
  so there exists an isolated  real root interval $\textbf{\emph{Z}}_i$ which contains the point $(\hat{\textbf{\emph{x}}}, \sqrt{-1/f(\hat{\textbf{\emph{x}}})})$.
  Hence the intersection of $\emph{\textbf{X}}$ and $\pi(\textbf{\emph{Z}}_i)$ is not an empty set, which contradicts to our suppose, so $f(\hat{\textbf{\emph{x}}})\geq0$. This proves the first part of the theorem.

On the other hand, assume $f(\hat{\textbf{\emph{x}}})\geq0$, then $g=y^2f(\hat{\textbf{\emph{x}}})+1\geq 1 $ for any real value of $y$,
so $\pi(\textbf{\emph{Z}}_i)$ does't contain $\hat{\textbf{\emph{x}}}$ for all $i\in \{1,\cdots,m\}$. This establishes that the intersection of $\textbf{X}$  and $\pi(\textbf{\emph{Z}}_i)(i=1,\cdots,m)$ is an empty set.
This completes the proof.  \end{proof}

 Similar to  theorem 3.1, if we construct $g=y^2f-1$, then we can decide whether $f(\hat{\textbf{\emph{x}}})\leq 0$ or not.

According to the above theorem, the following $Deter_{-}Sign$ can be used to determine the sign of $f(\hat{\textbf{\emph{x}}})$ if $0 \in [a,b]$.

\noindent \textbf{Algorithm 2}: $Deter_{-}Sign$

Input: A polynomial system \textbf{\emph{f}}; An isolated real root interval $\emph{\textbf{X}}$  of \textbf{\emph{f}}; A polynomial $f$, and an index $I$, the value of $I$ is 1 or -1.

Output: If $f(\hat{\textbf{\emph{x}}})I \geq 0$, then return 1; else return -1;
\begin{enumerate}
\item If I=1
   \begin{itemize}
   \item Let $g=y^2f+1$;
   \end{itemize}
\item Else
   \begin{itemize}
   \item Let $g=y^2f-1$;
   \end{itemize}
\item End If
\item $real_{-}roots=real_{-}root_{-}isolate([\textbf{\emph{f}},g])$;
\item $sum=0$;
\item For $i=1:length(real_{-}roots)$
\begin{itemize}

      \item If $\emph{\textbf{X}} \bigcap \pi (real_{-}roots\{i\})=\emptyset$,
               then $sum=sum+1$;
      \item Else,
               return -1;

      \item End If
\end{itemize}

\item End For
\item If $sum=nreal$
then return 1;
\item End If
\end{enumerate}

Using algorithm 2, it is easy to see that if
$$Deter_{-}sign(\textbf{\emph{f}},f,\emph{\textbf{X}},1)=1 \wedge Deter_{-}sign(\textbf{\emph{f}},f,\emph{\textbf{X}},-1)=1,$$ then $f(\hat{\textbf{\emph{x}}})= 0$.

In the following, a simple example is given to explain how the algorithm  determines the sign of $f(\hat{\textbf{\emph{x}}})$.

\begin{example}\label{example2}
 Let \textbf{{f}}$=[x^2+y-2,x+2y-3]$, and $h=3x+y-4$. It is easy to see that \textbf{{f}} and $h$ has a common point $(1,1)$.
 \end{example}

Using algorithm 1, we obtain an isolated real root interval of \textbf{\emph{f}} which contains the point $(1,1)$.
\vspace{-.1cm}
\[\begin{split}\emph{\textbf{X}}_1=[ & 0.999999999999900,   1.000000000000100],\\ [& 0.999999999999900  , 1.000000000000100].\end{split}\]

Substituting  $\emph{\textbf{X}}_1$ into $h$, we have \[h(\emph{\textbf{X}}_1)= 10^{-11}\times[-0.400124378074906 ,0.399680288865056].\]

Now, let $h_1=a^2(3x+y-4)+1$, then the isolated real root intervals of $[ \textbf{\emph{f}},h_1]$ are
\vspace{-.3cm}
\[\begin{split}\emph{\textbf{Z}}_1=[& -0.500000000000100 , -0.499999999999900],\\
[& 1.749999999999900  , 1.750000000000100],\\
[& -0.516397779494422 , -0.516397779494222].\end{split}\]
\vspace{-.2cm}
\[\begin{split}\emph{\textbf{Z}}_2=
[& -0.500000000000100 , -0.499999999999900],\\
[& 1.749999999999900  , 1.750000000000100],\\
[& 0.516397779494422 , 0.516397779494222].\end{split}\]

So $\emph{\textbf{X}}_1\cap \pi(\emph{\textbf{Z}}_1)=\varnothing $ and $\emph{\textbf{X}}_2 \cap \pi(\emph{\textbf{Z}}_2)=\varnothing$, then we have $h(1,1)\geq 0$. Similarly,
we obtain $h(1,1)\leq 0$ through isolating
the real roots of [ \textbf{\emph{f}}$,a^2(3x+y-4)-1]$. Hence $h(1,1)=0$.

Now, we give the following algorithm for computing the isolated intervals which satisfy the inequations \textbf{\emph{h}} in $(4)$.

\noindent \textbf{Algorithm 3:} $Dele_{-}Inequ$

Input: Polynomial system \textbf{\emph{f}} and  \textbf{\emph{h}};  isolated real root intervals $real_{-}roots$ of \textbf{\emph{f}}.

Output: $real_{-}roots=\{\emph{\textbf{X}} \in real_{-}roots | \emph{\textbf{h}}(\hat{\textbf{\emph{x}}})\neq0, \textbf{\emph{f}} (\hat{\textbf{\emph{x}}})=0, $
$ \hat{\textbf{\emph{x}}} \in \emph{\textbf{X}} \} $.
\begin{enumerate}
\item Set $Index=\{\}$;
\item For {$s=1:length(real_{-}roots)$}
     \begin{itemize}
           \item  Substituting $real_{-}roots\{s\}$ into  \textbf{\emph{h}}, and suppose \\ $ \textbf{\emph{h}}(real_{-}roots\{s\})=([a_1,b_1],\cdots,[a_t,b_t])$;\
           \item Let $set=\{i \in \{1,\cdots,t\} | 0 \in [a_i,b_i]\}$;
           \item If $length(set)\neq t$, then break;
           \item End if
           \item For $t=1:length(set)$
                 \begin{itemize}
                       \item $Sign_1=Deter_{-}sign(\textbf{\emph{f}},real_{-}roots\{s\}, h_t,1)$;
                        \item  If $sign_1=1$,
                                then \\
                                $Sign_2=Deter_{-}sign(\textbf{\emph{f}},real_{-}roots\{s\}, h_t,-1)$;
                         \begin{itemize}
                              \item If $sign_2=1$,  then $Index=Index \bigcup \{s\}$, break;


                        \item End If
                        \end{itemize}

                        \item End If

                 \end{itemize}
          \item End For
     \end{itemize}
     \item End For;
     \item $real_{-}roots=real_{-}roots \setminus real_{-}roots(Index).$
\end{enumerate}

In the following, we will present a method to deal with the non-negative systems \textbf{\emph{n}} in \emph{semi-algebraic} system defined in (4).

\noindent \textbf{Algorithm 4 :} $Dele_{-}Nonnega$

Input: Two Polynomial systems: \textbf{\emph{f}} and \textbf{\emph{n}}; isolated real root intervals of \textbf{\emph{f}}: $real_{-}roots$ .

Output: $real_{-}roots=\{\emph{\textbf{X}} \in real_{-}roots | \textbf{\emph{n}}(\hat{\textbf{\emph{x}}})\geq0, \textbf{\emph{f}} (\hat{\textbf{\emph{x}}})=0, $ $ \hat{\textbf{\emph{x}}} \in \emph{\textbf{X}} \} $.

\begin{enumerate}
\item Set $Index=\{\}$;
\item For $t=1:length(real_{-}roots)$
     \begin{itemize}
     \item Substituting $real_{-}roots\{t\}$ into \textbf{\emph{n}}, and suppose \\ \textbf{\emph{n}}$(real_{-}roots\{t\})=([a_1,b_1],\cdots,[a_s,b_s])$;
     \item If there exists some $i \in \{1,\cdots,s\}$ such that $b_i<0$, then let $Index=Index \cup \{t\}$;

     \item Else if there exists some $i \in \{1,\cdots,s\}$ such that $0 \in [a_i,b_i]$.
       Let $set=\{i \in \{1,\cdots,s\} | 0 \in [a_i,b_i] \}$;
     \begin{itemize}
           \item For $j=1:length(set)$
                 \begin{itemize}
                       \item $Sign_1=Deter_{-}sign(\textbf{\emph{f}},real_{-}roots\{t\},n_j,1)$;
                       \item If $sign_1=-1$, then  $Index=Index \cup \{s\}$, break;

                       \item End If

                 \end{itemize}
           \item End for
     \end{itemize}
     \item End if
     \end{itemize}
     \item End for
     \item $real_{-}roots=real_{-}roots \setminus real_{-}roots(Index).$
\end{enumerate}

The following algorithm is to remove the intervals which satisfy \textbf{\emph{p}} in $(4)$.

\noindent \textbf{Algorithm 5: } $Dele_{-}Posi$

Input:  Two polynomial systems: \textbf{\emph{f}} and  \textbf{\emph{p}}; isolated intervals $real_{-}roots$ of  \textbf{\emph{f}};

Output: $real_{-}roots=\{\emph{\textbf{X}} \in real_{-}roots | \textbf{\emph{p}}(\hat{\textbf{\emph{x}}})>0, \textbf{\emph{f}} (\hat{\textbf{\emph{x}}})=0, $where$ \hat{\textbf{\emph{x}}} \in \emph{\textbf{X}} \} $.
\begin{enumerate}
\item $real_{-}roots=Dele_{-}Inequ(\emph{\textbf{f}},real_{-}roots,\emph{\textbf{p}})$;
\item $real_{-}roots=Dele_{-}Nonnega(\emph{\textbf{f}},real_{-}roots,\emph{\textbf{p}})$;
\end{enumerate}

Until now, we have described all the parts of the algorithm for isolating the real roots of \emph{semi-algebraic} system. Here, we present the whole algorithm in the following.

\noindent \textbf{Algorithm 6: } $real_{-}root_{-}semi$

Input: A \emph{semi-algebraic} system in the form of $(4)$.

Output: isolated real root intervals of $(4)$ or return $\{\}$;

\begin{enumerate}
\item $real_{-}roots=real_{-}root_{-}isolate(\textbf{\emph{f}})$;
\item If {$length(real_{-}roots)=0$}, then return  $\{\}$;
     \item  Else
     \begin{itemize}
     \item $real_{-}roots=Dele_{-}Inequ(\textbf{\emph{f}},real_{-}roots,\textbf{\emph{h}})$;
     \item If $length(real_{-}roots)=0$, then return  $\{\}$;
           \item Else
           \begin{itemize}
                 \item $real_{-}roots=Dele_{-}Nonnega(\textbf{\emph{f}},real_{-}roots,\textbf{\emph{n}})$;
                 \item If {$length(real_{-}roots)=0$}, then return  $\{\}$;
                 \item Else
                 \begin{itemize}
                 \item $real_{-}roots=Dele_{-}Posi(\textbf{\emph{f}},real_{-}roots,\textbf{\emph{p}})$;\
                 \item If $length(real_{-}roots)=0$, then return $\{\}$;
                 \item End If
                 \end{itemize}
                 \item End If
           \end{itemize}
           \item End If

     \end{itemize}

\item End If
\item Return $real_{-}roots$.
\end{enumerate}

In the following, a simple example is provided to illustrate algorithm 6 for isolating the real roots of a \emph{semi-algebraic} system.
\begin{example}\label{example1}
 Given a {semi-algebraic} system,
\begin{eqnarray}
SAS=\left\{ {\begin{array}{*{20}{c}}
 f_1&=5+13x_1-10x_2-82x_1^2+71x_1x_2+16x_2^2. \hfill  \nonumber \\
 f_2&=403.22x_1-314.64+73.16x_1^2-269.26x_1x_2+300.96x_2-\hfill\\& 48x_1^3+53x_1^2x_2-28x_2^2x_1+95.76x_2^2. \hfill \nonumber \\
  & x_1-0.3>0,x_2>0,x_1-3.42\neq 0.\hfill
  \end{array}} \right.
\end{eqnarray}
\end{example}
Step 1: Denote the isolated real root intervals of $[f_1,f_2]$ by $\textbf{\texttt{X}}$, using algorithm 1, we obtain six intervals:

\vspace{-.6cm}
 \[ \begin{split}\emph{\textbf{X}}_1=[& 3.419999999999901 ,  3.420000000000101] , \\ [& 3.202328957744943 ,  3.202328957745143];\end{split}\]
\vspace{-.3cm}
\[\begin{split}\emph{\textbf{X}}_2=
[ &3.419999999999901 ,  3.420000000000101],\\
 [& -17.753578957745148 ,-17.753578957744949];\end{split}\]
\vspace{-.2cm}
  \[\begin{split}\emph{\textbf{X}}_3= [ &-0.099812907756173,  -0.099812907755973],\\ [& 0.857641152899082   ,0.857641152899282];\end{split}\]
\vspace{-.2cm}
\[\begin{split}\emph{\textbf{X}}_4=[ &-8.128471753948865 , -8.128471753948666],\\ [  &-7.758843802156950 , -7.758843802156750];\end{split}\]
\vspace{-.2cm}
  \[\begin{split}\emph{\textbf{X}}_5=
[&0.584413218566582  , 0.584413218566782],\\
 [&   -2.373999555338525 , -2.373999555338325];\end{split}\]
\vspace{-.2cm}
\[\begin{split}\emph{\textbf{X}}_6=[ & 0.560588744512268,   0.560588744512469],\\
[ & 0.376338245290841  , 0.376338245291041].\end{split}\]

\label{}

Step 2: Substituting $\textbf{\texttt{X}}$ into $x_1-3.42$, we have

$ev_1=[-0.000000000000099  , 0.000000000000100]$,

$ev_2=[ -0.000000000000099,   0.000000000000101]$,

$ev_3=[ -3.519812907756173 , -3.519812907755973]$,

$ev_4=[ -11.548471753948863 ,-11.548471753948665]$,

$ev_5=[ -2.835586781433418  ,-2.835586781433218]$,

$ev_6=[ -2.859411255487732 , -2.859411255487531]$.

It is easy to see that $x_1-3.42\neq0$ at any point contained in intervals $\emph{\textbf{X}}_3,\emph{\textbf{X}}_4,\emph{\textbf{X}}_5,\emph{\textbf{X}}_6$,
so these intervals should be retained.
Using algorithm 3 to process intervals $\emph{\textbf{X}}_1$ and $\emph{\textbf{X}}_2$, we know that $x_1-3.42$ is equal to zero at the solutions of $\{f_1,f_2\}$, so we discard $\emph{\textbf{X}}_1$ and $\emph{\textbf{X}}_2$ at this step.

Step 3:
Substituting $\emph{\textbf{X}}_3$ $\emph{\textbf{X}}_4$ $\emph{\textbf{X}}_5$ and $\emph{\textbf{X}}_6 $ into $\textbf{\emph{P}}=\{x_1-0.3,x_2\}$, we obtain the following four intervals
\[\begin{split}\textbf{\emph{P}}(\emph{\textbf{X}}_3)=
[&-0.399812907756173 , -0.399812907755973],\\
[ & 0.857641152899082  , 0.857641152899282];\end{split}\]
\vspace{-0.2cm}
\[\begin{split}\textbf{\emph{P}}(\emph{\textbf{X}}_4)=
[& -8.428471753948864,  -8.428471753948665],\\
[ & -7.758843802156949  ,-7.758843802156748];\end{split}\]
\[\begin{split}\textbf{\emph{P}}(\emph{\textbf{X}}_5)=
[ &0.284413218566582  , 0.284413218566782],\\
[ & -2.373999555338525 , -2.373999555338325];\end{split}\]
\vspace{-0.1cm}
\[\begin{split}\textbf{\emph{P}}(\emph{\textbf{X}}_6)=
[&0.260588744512268 ,  0.260588744512469],\\
 [&0.376338245290841   ,0.376338245291041].\end{split}\]
where $\textbf{\emph{P}}(\emph{\textbf{X}}_i)$ means \emph{ interval extension} $\textbf{\emph{P}}$ at $\emph{\textbf{X}}_i,i=3,4,5,6.$
Obviously, $\emph{\textbf{X}}_3$, $\emph{\textbf{X}}_4$ and $\emph{\textbf{X}}_5$ should be deleted from $\textbf{\texttt{X}}$.

Therefore, the isolated real root interval of $SAS$ is
\vspace{-0.1cm}
\[\begin{split}\emph{\textbf{X}}_6=
[& 0.560588744512268  , 0.560588744512469],\\
[ & 0.376338245290841   ,0.376338245291041].\end{split}\]
\subsection{Comparison experiments}

 Algorithm 6 has been implemented in  Matlab 2010a  named $real_{-}roots_{-}semi$. In this subsection,  we do some experiments (see Appendix A) to compare it with Maple package
Discover described in \cite{XL2002} on the platform of Maple 15 classic work sheet. All experiments are carried out on a Dell PC with Intel Core i3-2120 at 3.30GHZ and 4GB of RAM in Windows 7. 


\begin{table}[!hbp]
\begin{tabular}{|c|c|c|c|c|c|}
\hline
Example No.& Max degree &$N_{-}v$ &$N_{-}p$ & Discover & $real_{-}roots_{-}semi$  \\
\hline
A.1 &2&3&1& 0.062&  0.259715 \\
\hline
A.2 &3&3& 1&0.483 & 0.188181  \\
\hline
A.3 & 4&3&0&219.759 & 0.461846  \\
\hline
A.4 &2&4&1& 13.947 &  0.302547\\
\hline
A.5 &3&4&1& 0.31 &  0.756411\\
\hline
A.6 &4&4&3& 9.734 & 6.106957 \\
\hline
A.7&2& 4&2&13.821 &1.942310\\
\hline
A.8 &4&5&5& $\infty$ &0.794765 \\
\hline
A.9 &2&6&5& $\infty$ & 11.272852 \\
\hline
\end{tabular}
\caption{Time comparison, units:s.}\label{Tab:1}
\end{table}
Table 1 gives the comparison of execution time. In the first row, $M_{-}d$, $N_{-}v$ and $N_{-}p$ denote the max degree, number of variables, number of real points of the
\emph{semi-algebraic} system. And $\infty$ means the program out of memory. Among these examples, A.5 is from \cite{XZ2006}, the other examples are chosen randomly.
As shown in Table 1, $Discover$ is faster than  $real_{-}root_{-}semi$ for  Example A.1 and A.5, since these \emph{semi-algebraic} system are simple. When the system becomes more complicated,
$real_{-}root_{-}semi$ is more efficient than $Discover$. From table 1 we can also find that the costs of example A.6, A.7 and A.9 are much higher than those of other examples compared
 because these examples can satisfy the inequations. 
 \section{Applications}

Most engineering problems can be reduced to solving nonlinear equations. In this section, some applications from robot kinematic, piecewise algebraic curve, chemical engineering,
 circuit design are tested to show the efficiency of our algorithm.

\subsection{Polynomial system}
In this subsection, we first investigate applications of our algorithm to some polynomial systems.
\begin{example}
 Production of synthesis gas in an adiabatic reactor\cite{KS1988}.
\begin{eqnarray}
\left\{ {\begin{array}{*{20}{c}}
x_1x_7+2x_2x_7+x_3x_7-2x_6=0. \hfill   \\
x_7(x_3+x_4+2x+5)-2=0.\hfill \\
7x_1+7x_2+7x_5-1=0.\hfill\\
x_1+x_2+x_3+x_4+x_5-1=0.\hfill\\
400x_1x_4^3-178370x_3x_5=0.\hfill\\
x_1x_3-2.6058x_2x_4=0.\hfill\\
-28837x_1x_7-139009x_2x_7-78213x_3x_7+18927x_4x_7+8427x_5x_7+\hfill \\ 13492 -10690x_6=0.\hfill\\
 0 \leq x_i\leq 1, i=1,2,3,4,5.\hfill \\
 0 \leq x_i\leq 5 , i=6,7.\hfill
 \end{array}} \right.
\end{eqnarray}
\end{example}
The above system of equations represents three atom balances, a mole fraction
constraint, two equilibrium relations, and an energy balance equations. We obtain 8 real roots of this equation through 0.536665 seconds and 8 interval iterations  for the 8 solution, then remove 7 real roots after 0.061837 seconds. Let the midpoint of interval be the approximate root of the problem, then

$\textbf{\emph{x}}=[0.322870839476541,0.009223543539188,0.046017090960632,$\\$0.618171675070824, 0.003716850952815,0.576715395935549,2.977863450791145]$.

\begin{example}
Robot kinematic problem \cite{MF1995}.
\begin{eqnarray}
\left\{ {\begin{array}{*{20}{c}}
(0.004731x_3-0.1238)x_1-(0.3578x_3+0.001637)x_2+x_7-0.9338x_4\hfill \\-0.3571=0. \hfill   \\
0.2238x_1x_3+0.7623x_2x_3+0.2638x_1-x_7-0.07745x_2-0.6734x_4\hfill\\-0.6022=0.\hfill \\
x_6x_8+0.3578+0.004731x_2=0.\hfill\\
-0.7623x_1+0.2238x_2+0.3461=0.\hfill\\
x_1^2+x_2^2-1=0.\hfill\\
x_3^2+x_4^2-1=0.\hfill\\
x_5^2+x_6^2-1=0.\hfill\\
x_7^2+x_8^2-1.\hfill \\
x_i\geq -1, x_i \leq 1 ,i=1,2,\cdots,8.\hfill
 \end{array}} \right.
\end{eqnarray}
\end{example}
All 16 real solutions were founded in 0.826805 seconds of CPU time, which is more efficient than  3.783 seconds CPU time in \cite{MVP2010} to solve the problem.
\begin{example}
This example addresses the equilibrium of the products of a hydrocarbon
combustion process \cite{MM1990}. The problem is reformulated in the `element variables' space.
\begin{eqnarray}
\left\{ {\begin{array}{*{20}{c}}
y_1y_2+y_1-3y_5=0. \hfill   \\
2y_1y_2+y_1+3R_{10}y_2^2+y_2y_3^2+R_{7}y_2y_3+R_{9}y_2y_4+R_{8}y_2-Ry_5=0.\hfill \\
2y_2y_3^2+R_7y_2y_3+2R_{5}y_3^2+R_{6}y_3-8y_5=0.\hfill\\
R_{9}y_2y_4+2y_4^2-4Ry_5=0.\hfill\\
y_1y_2+y_1+R_{10}y_2^2+y_2y_3^2+R_{7}y_2y_3+R_{9}y_2y_4+R_{8}y_2+R_{5}y_3^2\hfill \\+R_{6}y_3+y_4^2-1=0.\hfill\\
y_i-0.0001\geq 0, 100-y_i\geq 0 ,i=1,2,\cdots,5.\hfill
 \end{array}} \right.
\end{eqnarray}
\end{example}
The value of parameter $R_i,i=1,2,\cdots,10$ can be found in \cite{MM1990}. The method in \cite{MVP2010} finds the single solution after 217.7 seconds of CPU time.
Using our method, we only use 0.2964 seconds CPU time.

\begin{example}
This example for computing the real roots of piecewise algebraic curve is modified from \cite{W2012}.
 \end{example}
 \vspace{-0.2cm}
\begin{figure}
\centering\includegraphics[width=110mm]{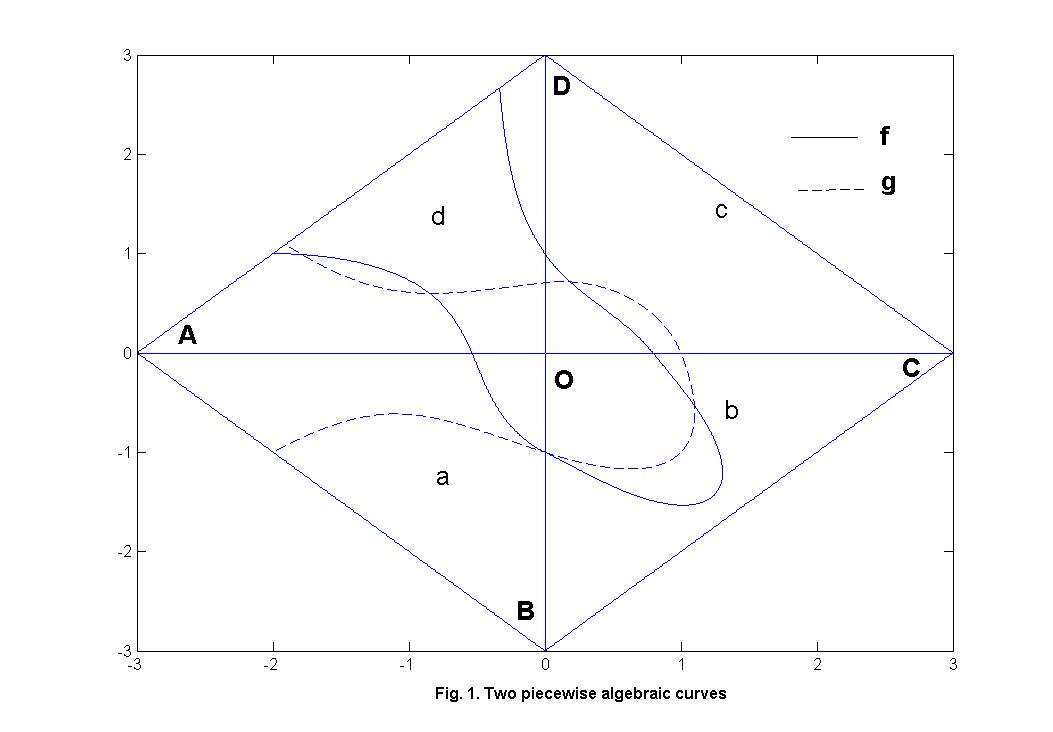}\\
\end{figure}

Let $\Delta=\{a,b,c,d\}$ be a regular triangulation of rectangular domain $ABCD$ in $\mathbb{R}^{2}$, where $a=[AOB]$, $b=[BOC]$, $c=[COD]$ and $d=[AOD]$,
where $A=(-3,0)$, $B=(0,-3)$, $C=(3,0)$ $D=(0,3)$ and $O=(0,0)$, $f$ and $g$ are algebraic curves (See Fig. 1).

Suppose that $f,g\in S_3^{1}(\Delta)$ and
\begin{equation}
\text{on cell $a$}:
\left\{ {\begin{array}{*{20}{c}}
f_1|_{a}=x^3-y^3+3y^2+xy-3x-4. \hfill  \nonumber \\
g_1|_{a}=x^3-y^3+2yx+4x^2-1. \hfill \\
 \end{array}} \right.
\end{equation}

\begin{equation}
\text{on cell $b$}:
\left\{ {\begin{array}{*{20}{c}}
f_1|_{b}=f_1|_{a}+x^2(2x+3y+4). \hfill  \nonumber \\
g_1|_{b}=g_1|_{a}+x^2(x+y-4).\hfill \\
 \end{array}} \right.
\end{equation}

\begin{equation}
\text{on cell $c$}:
\left\{ {\begin{array}{*{20}{c}}
f_1|_{c}=f_1|_{b}+y^2(x+y+5). \hfill  \nonumber \\
 g_1|_{c}=g_1|_{b}+y^2(2x+y+1).\hfill \\
 \end{array}} \right.
\end{equation}

\begin{equation}
\text{On cell $d$}:
\left\{ {\begin{array}{*{20}{c}}
 f_1|_{d}=f_1|_{a}+y^2(x+y+5). \hfill  \nonumber \\
 g_1|_{d}=g_1|_{a}+y^2(2x+y+1).\hfill \\
 \end{array}} \right.
\end{equation}

We take the cell $a$ for instance, isolating the  intersection points of $\{f_1|_{a},g_1|_{a}\}$ lying in $a$ can be converted to isolating the real roots of the following
\emph{semi-algebraic} system:
\begin{eqnarray}
\left\{ {\begin{array}{*{20}{c}}
x^3-y^3+3y^2+xy-3x-4=0. \hfill   \\
x^3-y^3+2yx+4x^2-1=0.\hfill \\
-x\geq0,-y\geq0,x+y+3\geq0\hfill\\
\end{array}} \right.
\end{eqnarray}
Using the algorithm $real_{-}root_{-}semi$, we know that there is only one real root lying in $a$. Similarly, we conclude that $\{f_1|_{c},g_1|_{c}\}$ has one real point in the interior
 of  $c$, and respectively, $\{f_1|_{b},g_1|_{b}\}$ and $\{f_1|_{d},g_1|_{d}\}$ has two common points in the cell of $b$ and  $d$.

Let the polynomial $\{f_1|_{b},g_1|_{b}\}$ be inequations, and append it to system (9), we conclude that $a$ and $b$ have a common real point. Hence,
$f$ and $g$ have 5 real points. Methods given by Wang, et al., based on interval method, Groebner basis in \cite{WZ2008,ZW2009,W2012,LW2009,WW2008} can also find the commom points of
piecewise algebraic curve, but they did not consider these points which lie in the boundary.

\subsection{Transcendental functions system}
In practice, systems of transcendental functions  appear in many applications \cite{KS1988,MF1995}, e.g., logarithm functions, exponential functions and trigonometric functions.

As far as we know, there exists three kinds of numerical methods to solve nonlinear equations: homotopy continuation method, Newton's iteration and interval bisection method. Each kind of algorithm has
its advantages and disadvantages. Homotopy continuation method is a global convergence algorithm, but it can only be used for solving polynomial equations. Given an appropriate initial point,
the convergence of Newton's method is typical local quadratic, but how to choose the initial point is a difficult problem. The interval method
is unacceptable in computation when some conditions are complicated, such as more variables,  the huge width of the initial interval and the more accuracy of the final result.

The basic idea of our method is to replace the exponential function by its Taylor expansion first. In this step, rather than using the initial interval, we reduce the width
of the initial interval using interval algorithm, until the errors between the  exponential function and its Taylor expansion is smaller than a tolerance.
   We then apply homotopy method to the polynomial system to obtain the approximate real roots.
  At last, the real roots of transcendental functions are found based on these approximate real roots using the Newton's method.

  In the following,   we apply this idea to solve a  circuit design problem from \cite{RR1993}.

\begin{example}
 Circuit design problem with extraordinary sensitivities to small perturbations.

\begin{eqnarray}
\left\{ {\begin{array}{*{20}{c}}
(1-x_1x_2)x_3(e^{(x_5(g_{1k}-g_{3k}x_710^{-3}-g_{5k}x_810^{-3}))}-1)-g_{5k}+g_{4k}x_2=0. \hfill   \\
(1-x_1x_2)x_4(e^{(x_6(g_{1k}-g_{2k}-g_{3k}x_710^{-3}+g_{4k}x_910^{-3}))}-1)-g_{5k}x_1+g_{4k}=0.\hfill \\
x_1x_3-x_2x_4=0.\hfill\\
k=1,2,3,4.\hfill\\
x_i \in [0,10], i=1,\cdots,9.\hfill
 \end{array}} \right.
\end{eqnarray}
\end{example}
The constant $g_{ik}$ are given by the following matrix
\[\left(
  \begin{array}{cccc}
    0.4850 & 0.7520 & 0.8690 & 0.9820 \\
    0.3690 & 1.2540 & 0.7030 & 1.4550 \\
    5.2095 & 10.0677 & 22.9274 & 20.2153 \\
    23.3037 & 101.7790 & 111.4610 & 191.2670 \\
    28.5132 & 111.8467 & 134.3884 & 211.4823 \\
  \end{array}
\right),\]

Step 1: Since the exponential function in system (10) only contains variables $x_i, i=5,6,7,8,9$,  we only bisect the intervals of these variables by interval Newton iteration.
This is different from other algorithms. After 10 times interval iterations for each variable, we have the last interval of these variables
\begin{center}
$[ 7.998046875 ,  8.007812500]$,

  $[ 7.998046875 ,  8.007812500],$

   $[5.000000000 ,  5.009765625],$

   $[0.996093750 ,  1.005859375],$

   $[1.992187500  , 2.001953125].$
   \end{center}
   So we only need 50 bisections, and this progress takes about 9 seconds.
Now, we compute the \emph{ interval extension} of $g_{1k}=x_5(g_{1k}-g_{3k}x_710^{-3}-g_{5k}x_810^{-3})$ and $g_{2k}=x_6(g_{1k}-g_{2k}-g_{3k}x_710^{-3}+g_{4k}x_910^{-3})$,
and using the Taylor expansion of the exponential function in the midpoint of the \emph{ interval extension} of $g_{1k}$ and $g_{2k}$ to replace
$e^{g_{1k}}$ and $e^{g_{2k}}$ for $k=1,2,3,4$ in equation (10).  At this point, we use the second order Taylor expansion for $e^{g_{14}}$ and the first order for other exponential
functions. Note that the errors generated by this step are not more than $10^{-2}$.

Step 2: For the polynomial system obtained from step 1, there exists 37 homotopy paths, but only $4$ roots are founded using homotopy algorithm, and  only 2 of them are real roots. This process takes 1.469615 seconds.

Step 3: Let the two real points be the initial point of the Newton's iteration for system $(10)$. After $0.092$ seconds we find that this two points converge to the same point.

\vspace{-.6cm}
\[
\footnotesize\begin{split}\hat{\emph{\textbf{x}}}=[&.899999952618, .449987471886, 1.00000648241, 7.99997144063, 5.00003127610, \\&.999987723423, 2.00006854190, 7.99969268290, 2.00005248366].\end{split}\]

\vspace{-.2cm}
Compared with $\hat{\emph{\textbf{x}}}$, we find that the errors between $\hat{\emph{\textbf{x}}}$ and the initial point from step 2 is 1.2602 and 0.0017, respectively.
The last one is close to $\hat{\emph{\textbf{x}}}$ which ensure the convergence of the Newton iterative algorithm in this step.

 From the above description, we totally  use  about 11 seconds to find  the real point of equation (10), which is much less than 436.5 seconds in \cite{MVP2010}.


\section{Conclusion}
For a class of \emph{semi-algebraic} system, this paper presents a numerical method for isolating the real roots. The algorithm first obtains the isolated intervals of a zero-dimensional
 polynomial system using hybrid technique \cite{S2012}, and substitute it into the inequations, nonnegative system, and positive systems to remove some intervals.
 Such implementation will be efficient if the interval extension of constrained equation does not contain zero at the isolated real root interval of the system.
 Otherwise, we give a complete numerical algorithm to determine if a polynomial equation is equal to zero. We implement our algorithm in Matlab environment. Many random
 examples have been checked along with comparison to famous software Discover . At last, this new method is used to solve problems originated from several applications.
For transcendental functions, we give an idea for solving this problem which is a combination of  Newton's method, homotopy algorithm and interval arithmetic.
However, the number of bisections, the order of Taylor expansion and the error control are deserved to further study.

\section{Acknowledge}
The authors especially thank associate professor Jinming Wu
for the help of discussion about the piecewise algebraic curve.





\bibliography{<your-bib-database>}



\appendix
\renewcommand\thesection{\appendixname~\Alph{section}}
\renewcommand\theequation{\Alph{section}.\arabic{equation}}
\section{}
\begin{equation}
\left\{ {\begin{array}{*{20}{c}}
 -3+95x_3+29x_1x_2+63x_1x_3+37x_2^2-36x_3^2=0. \hfill \\
-2x_1+46x_2+4x_1x_2+50x_1x_3+61x_3x_2-99x_3^2=0. \hfill\\
 72-49x_1+13x_2-15x_1^2-12x_2^2+23x_3^2=0.\hfill\\
 3x_2+2x_1-1\geq 0, 3x_1-4x_2-3>0.\hfill\\
 \end{array}} \right.
\end{equation}

\begin{equation}
\left\{ {\begin{array}{*{20}{c}}
 -85x_1-76x_2+75x_3-41x_2^2-84x_2^3=0. \hfill  \\
-39+26x_1x_3+47x_1^3-47x_1^2x_2+91x_2^2x_1+43x_3x_1x_2=0 .\hfill  \\
16x_1+25x_3+13x_1x_2+x_2^2-8x_3x_1^2-74x_2x_3^2=0\hfill \\
19x_3+10x_1x_3+53x_3x_2-97x_3x_1^2+57x_2^2x_1+68x_2x_3^2 \geq0 \hfill \\
47x_1x_2+58x_3x_2+30x_3^2+55x_3x_1^2+56x_3x_1x_2+55x_3^3 \geq0\hfill\\
 48x_2+34x_1^2+32x_1^3-41x_2^2x_1-56x_2^3+63x_2x_3^2 \geq 0\hfill\\
27+11x_1+46x_3-42x_1x_3-60x_1^3-48x_2^2x_3 \geq 0 \hfill
 \end{array}} \right.
\end{equation}

\begin{equation}
\left\{ {\begin{array}{*{20}{c}}
 -49x_2+31x_3x_2+73x_1x_3^2+95x_2^2x_1^2+68x_1x_2x_3^2-29x_1x_3^3=0. \hfill  \\
37+5x_1^2-36x_3^2-57x_2x_3^2+85x_1x_3^3+80x_2^2x_3^2=0 .\hfill  \\
30x_2^2-3x_3x_2-56x_1^2x_2-91x_2^2x_1^2-70x_3x_2^2x_1+42x_2^4=0\hfill \\
62+96x_2-51x_3+89x_1^2+14x_2^2x_1-79x_2^2x_3 >0 \hfill \\
86+57x_1-35x_1^2+57x_2^2+28x_3^2+63x_3x_1^2 >0\hfill\\
 -61x_1+45x_3-50x_2^2-22x_1^3+48x_1^2x_2-82x_2x_3^2 > 0\hfill\\
-85x_3x_2-44x_3x_1^2-31x_2^2x_1+45x_1x_3x_2+49x_2^3-58x_3^3 \geq 0 \hfill
 \end{array}} \right.
\end{equation}

\begin{equation}
\left\{ {\begin{array}{*{20}{c}}
 94-99x_1x_3+46x_2x_4+41x_3^2+95x_3x_4=0. \hfill  \\
 -59x_1-15x_4+30x_1x_2-57x_1x_3+50x_2^2+62x_2x_3=0 .\hfill  \\
53-9x_1-99x_4+12x_1x_4+20x_2x_4+99x_4^2=0\hfill \\
-82-22x_3-60x_2^2+3x_2x_4-83x_3x_4-84x_4^2=0 \hfill \\
-74x_2x_3+87x_3^2+68x_4x_1^2-88x_3x_2^2-88x_2x_3^2-87x_2x_3x_4 \geq 0 \hfill
 \end{array}} \right.
\end{equation}

\begin{equation}
\left\{ {\begin{array}{*{20}{c}}
 2-6.5x_1+x_1^2x_2-0.5x_3=0. \hfill   \\
6x_1-x_1^2x_2-5x_4+5x_2=0 .\hfill  \\
2-6.5x_3+x_3^2x_4-0.5x_1=0\hfill \\
6x_3-x_3^2x_4+1+0.5x_2-0.5x_4=0 \hfill \\
x_1>0,x_2>0,x_3>0,x_4>0\hfill
\end{array}} \right.
\end{equation}

\begin{equation}
\left\{ {\begin{array}{*{20}{c}}
 9x_1^2-20+x_1^2x_4-5x_4-x_1^4-4x_1^3x_2+20x_1x_2-4x_1^2x_2x_4+20x_2x_4\hfill \\+3x_1^2x_3^2-15x_3^2=0. \hfill   \\
3x_1^2x_2^2-39x_1^2+4x_1x_2^3-52x_1x_2-5x_3x_1x_2^2+65x_1x_3+3x_1x_2^2x_4 \hfill \\-39x_1x_4-3x_2^4+39x_2^2=0. \hfill\\
 -7-7x_2+4x_1x_3-2x_1x_4+5x_2x_3+5x_2x_4=0.\hfill\\
3+2x_1+3x_3-3x_1^2+5x_1x_2-x_3x_4=0.\hfill\\
 x_1^2-5\neq 0, x_2^2-13\neq 0.\hfill\\
 \end{array}} \right.
\end{equation}

\begin{equation}
\left\{ {\begin{array}{*{20}{c}}
2x_1^3-3x_1^2x_2-8x_1^2-5x_1x_2^2-12x_1x_2+5x_1^2x_3-12x_1x_3x_2-20x_1x_3 \hfill \\+4x_1^2x_4-12x_1x_2x_4-16x_1x_4-12x_2^3-16x_2^2-9x_2^2x_3-12x_2x_3=0. \hfill  \\
50x_3-65-20x_1x_3+26x_1-30x_1x_3x_2+39x_1x_2-70x_2x_3^2+91x_2x_3\hfill \\-70x_3^3+ 91x_3^2-40x_3^2x_4+52x_3x_4=0. \hfill\\
 5x_1-2x_3-2x_1x_4-2x_2x_4-x_3^2-7x_4^2=0.\hfill\\
-7-7x_1+2x_1x_2-7x_2x_4-3x_3x_4+5x_4^2=0.\hfill\\
 x_1-3x_2-4\neq 0,x_3-1.3\neq 0.\hfill\\
 \end{array}} \right.
\end{equation}

\begin{equation}
\left\{ {\begin{array}{*{20}{c}}
66x_4x_2^3x_1-44x_1^2x_2x_4+66x_1x_2x_4-165x_4x_5^2x_2^2+110x_4x_5^2x_1-165x_4x_5^2=0. \hfill   \\
-10x_1+62x_3-82x_4+80x_5-44x_2^2+71x_3x_4=0. \hfill\\
74x_1x_2+72x_1x_4+37x_2x_4-23x_2x_5+87x_3x_5+44x_4^2=0.\hfill\\
11-49x_2-47x_5+40x_1^2-81x_2x_3+91x_2x_4=0.\hfill\\
-28x_1+16x_2+30x_4-27x_1x_2-15x_1x_4-59x_3x_4=0.\hfill \\
3x_2^2-2x_1+3 \neq0,x_4\neq0. \hfill
 \end{array}} \right.
\end{equation}

\begin{equation}
\left\{ {\begin{array}{*{20}{c}}
36x_2+91x_3-22x_2x_3+51x_3x_6-27x_4x_5+50x_5^2=0. \hfill   \\
25x_1+31x_2-27x_4+65x_1x_2+88x_2x_5+10x_3x_5=0. \hfill\\
95x_2+68x_1x_2-29x_1x_6+5x_2x_4-26x_3^2-51x_3x_4=0.\hfill\\
5-36x_4-57x_1x_2+85x_3^2+80x_4x_6+90x_6^2=0.\hfill\\
  65x_1x_2-12x_1x_5+78x_1x_6+5x_2^2-63x_2x_6-5x_3x_5=0.\hfill \\
-70+42x_1+9x_1^2-21x_1x_2-27x_1x_5-79x_6^2=0. \hfill \\
x_4^2+x_5x_6-2 \geq 0.\hfill\\
3x_1-4x_2+4>0. \hfill
 \end{array}} \right.
\end{equation}

\end{document}